\documentclass[reqno, 12pt, a4letter]{amsart}
\usepackage{amsmath,amsxtra,amssymb,latexsym, amscd,amsthm}

\usepackage{epsfig}
\usepackage{graphics}
\usepackage{color}

\usepackage{geometry}

\numberwithin{equation}{section}
\theoremstyle{plain}
\newtheorem{thm}{Theorem}[section]

\newtheorem{prop}[thm]{Proposition}
\newtheorem{lemma}[thm]{Lemma}

\theoremstyle{definition}
\newtheorem{dfn}[thm]{Definition}
\newtheorem{ex}[thm]{Example}
\theoremstyle{remark}
\newtheorem{rem}[thm]{Remark}

\newcommand{\C}{\mathbb C}
\newcommand{\N}{\mathbb N}

\newcommand{\R}{\mathbb R}

\newcommand{\Z}{\mathbb Z}

\newcommand{\ve}{\varepsilon}

\def\dim{\operatorname{dim}}

\def\ees{{\accent"5E e}\kern-.385em\raise.2ex\hbox{\char'23}\kern-.08em}
\def\EES{{\accent"5E E}\kern-.5em\raise.8ex\hbox{\char'23 }}
\def\ow{o\kern-.42em\raise.82ex\hbox{
\vrule width .12em height .0ex depth .075ex \kern-0.16em \char'56}\kern-.07em}
\def\OW{O\kern-.460em\raise1.36ex\hbox{
\vrule width .13em height .0ex depth .075ex \kern-0.16em \char'56}\kern-.07em}

\begin{document}
\title[The bifurcation set of a real polynomial function]{The bifurcation set of a real polynomial function of two variables and Newton polygons of singularities at infinity}

\author{Masaharu Ishikawa$^\dagger$}
\address{Mathematical Institute, Tohoku University, Sendai, 980-8578, Japan}
\email{ishikawa@m.tohoku.ac.jp}

\author{Tat-Thang Nguyen$^\ddagger$}
\address{Institute of Mathematics, Vietnam Academy of Science and Technology, 18 Hoang Quoc Viet road, Cau Giay district, 10307 Hanoi, Vietnam}
\email{ntthang@math.ac.vn}

\author{Ti\EES N-S\OW N Ph\d{a}m$^*$}
\address{Department of Mathematics, University of Dalat, 1 Phu Dong Thien Vuong, Dalat, Vietnam}
\email{sonpt@dlu.edu.vn}

\thanks{$^{\dagger}$The first author is supported by the Grant-in-Aid for Scientific Research (C), JSPS KAKENHI Grant Number 16K05140}
\thanks{$^{\ddagger}$The second author is supported by the National Foundation for Science and Technology Development (NAFOSTED), Vietnam}
\thanks{$^{*}$The third author is supported by the National Foundation for Science and Technology Development (NAFOSTED), Grant number 101.04-2016.05, Vietnam}

\keywords{atypical value, bifurcation set, toric compactification}

\subjclass[2010]{Primary: 32S20, Secondary: 32S15, 32S30}

\begin{abstract}
In this paper, we determine the bifurcation set of a real polynomial function of two variables
for non-degenerate case in the sense of Newton polygons by using a toric compactification.
We also count the number of singular phenomena at infinity, called ``cleaving'' and ``vanishing''
in the same setting. Finally, we give an upper bound of the number of elements in the bifurcation set
in terms of its Newton polygon. To obtain the upper bound, we apply toric modifications to 
the singularities at infinity successively.
\end{abstract}

\maketitle

\section{Introduction}

Let $f:K^2\to K$ be a polynomial function, where $K$ is either $\C$ or $\R$. 
It is well-known that there exists a finite set $B \subset K$
such that $f:K^2\setminus f^{-1}(B)\to K\setminus B$ is a locally trivial fibration.
The smallest set of $B$ with the above properties is called the {\it bifurcation set},
which we denote by $B_f$.
Let $\Sigma_f$ denote the set of critical values of $f$.
Obviously, $\Sigma_f\subset B_f$. 
An element in $B_f$ caused by such a singular phenomenon at infinity 
is called an {\it atypical value} of $f$ at infinity.

There are many studies aiming to determine the bifurcation sets of polynomial functions.
The results of Suzuki~\cite{Suzuki1974}, Ha and Le~\cite{HL1984} and Ha and Nguyen~\cite{HaHV1989-2}
are known to be pioneering works in these studies,
where geometrical and topological characterizations of atypical values at infinity 
of complex polynomial maps are given.
The Newton polygon is one of the main tools in the study of atypical values at infinity,
for instance see~\cite{nz,lo,zaharia,ishikawa,thang}. 
Concerning real polynomial functions of two variables, Tib\u{a}r and Zaharia gave a characterization 
of the bifurcation set in~\cite{tz}
in terms of the first betti number, the Euler characteristic and vanishing and 
splitting phenomena of atypical fibers over the bifurcation set.
Real polynomial functions of two variables were studied by Coste and de la Puente more precisely
in~\cite{cp}, where they gave a characterization of bifurcation sets by using ``clusters'' and 
gave an algorithm to determine them. 
See~\cite{bp,hn,dt} for further studies related to this topic.

In this paper, we study the bifurcation sets of real polynomial functions
of two variables using Newton polygons, associated toric compactifications
and successive toric modifications.
These techniques were used by the first author in~\cite{ishikawa}
for determining the bifurcation sets of complex polynomial functions algorithmically.

To state our results, we prepare some terminologies.
Set $f(x,y)=\sum_{(m,n)} a_{m,n}x^my^n$, where $m,n\geq 0$.
Let $\varDelta(f)$ be the convex hull
of the integral points $(m,n)\in\R^2$ with $a_{m,n}\ne 0$. 
Note that we do not include the origin $(0,0)$ in the definition of $\varDelta(f)$ 
when $f(0,0)=0$, compare with~\cite{Kouchnirenko1976}.
A vector $P={}^t(p,q)\ne (0,0)$ consisting of coprime integers $p$ and $q$
is called a {\it primitive covector}.
For a given $P$, let $d(P;f)$ denote the minimal value of
the linear function $pX+qY$ for $(X,Y)\in\varDelta(f)$. 
Set $\varDelta(P;f):=\{(X,Y)\in\varDelta(f)\mid pX+qY=d(P;f)\}$, 
which is called a {\it face} of $\varDelta(f)$ if $\dim\varDelta(P;f)=1$.
The partial sum 
$f_P(x,y):=\sum_{(m,n)\in\varDelta(P;f)}a_{m,n}x^my^n$
is called the {\it boundary function} for the covector $P$.
If $\varDelta(P;f)$ is a face then it is called the {\it face function}.
Let $\Gamma_\infty^+(f)$ 
(resp. $\Gamma_\infty^0(f)$, $\Gamma_\infty^-(f)$) denote
the set of faces $\varDelta(P;f)$ of $f$ such that 
$P={}^t(p,q)$ satisfies either $p<0$ or $q<0$ and 
satisfies $d(P;f)>0$ (resp. $d(P;f)=0$, $d(P;f)<0$).
For a set $\Gamma(f)$ of faces of $\varDelta(f)$,
we say that $f$ is {\it non-degenerate} on $\Gamma(f)$ if the system of equations 
$\frac{\partial f_P}{\partial x}=\frac{\partial f_P}{\partial y}=0$ 
has no solutions in $(\R\setminus\{0\})^2$ for 
any face $\varDelta(P;f)$ in $\Gamma(f)$.

A face $\varDelta(P;f)$ in $\Gamma_\infty^0(f)$ 
is called a {\it bad face}. The face function on a bad face is given as
\begin{equation}\label{eq1.1}
   f_P(x,y)=b_P(t(x,y)),\quad t(x,y) = x^{|q|}y^{|p|},
\end{equation}
where $P={}^t(p,q)$ and $b_P$ is a polynomial of one variable $t$.
We say that $f_P$ is {\it Morse} if $b_P(t)$ is a Morse function on $\R\setminus\{0\}$ (i.e., it has only non-degenerate critical points on $\R\setminus\{0\}$).

\begin{thm}\label{thm1.1}
Suppose that $f(x,y)$ is non-degenerate 
on $\Gamma_\infty^+(f)\cup\Gamma_\infty^-(f)$
and the face function on any bad face is Morse.
Then, $c\in B_f$ if and only if one of the following holds: 
\begin{itemize}
\item[(i)] $c\in\Sigma_f$;
\item[(ii)] $c=f(0,0)$ and there exists 
$\varDelta(P;\tilde f)\in\Gamma_\infty^+(\tilde f)$
such that $\tilde f_P(x,y)=0$ has a solution 
in $(\R\setminus\{0\})^2$, where $\tilde f(x,y)=f(x,y)-f(0,0)$;
\item[(iii)]
$c$ is a critical value of $b_P|_{\R\setminus\{0\}}$ in~\eqref{eq1.1} 
for a bad face $\varDelta(P;f)$.
\end{itemize}
\end{thm}

Remark that the conditions in Theorem~\ref{thm1.1} are satisfied for generic choice of coefficients of $f$.

It is known in~\cite{tz, cp} that the value $c\in B_f$ is characterized by
the existence of a cleaving or vanishing family whose limit is $f=c$.
The precise definitions of these families are given in Section~2.

In the next theorem, we determine the number of cleaving and vanishing families.
For each $\varDelta(P;f)\in\Gamma_\infty^+(f)$, let $r^+(P;f)$ denote the number of 
non-zero real roots of $g_i(v_i)=0$ in~\eqref{eq2.1} below. 
For each bad face $\varDelta(P;f)\in\Gamma_\infty^0(f)$, 
let $r^0(P;f)$ denote the number of non-zero real roots of $\frac{db_P}{dt}(t)=0$.
Let $\text{\rm cleav}(f)$ and $\text{\rm vanish}(f)$ denote the numbers of cleaving families 
and vanishing families of $f$, respectively.

\begin{thm}\label{thm4.1}
Suppose that $f(x,y)$ satisfies the conditions in Theorem~\ref{thm1.1}.
Suppose further that $f$ has only isolated singularities.
Then
\[
\begin{split}
&\text{\rm cleav}(f)+\text{\rm vanish}(f)=2(R^++R^0) 
\quad \text{and}\quad 
0\leq \text{\rm vanish}(f)\leq 2R^0,
\end{split}
\]
where
\[
R^+=\sum_{\varDelta(P;f)\in\Gamma_\infty^+(f)}r^+(P;f),\quad
R^0=\sum_{\varDelta(P;f)\in\Gamma_\infty^0(f)}r^0(P;f).
\]
In particular, if there is no bad face then there is no vanishing family.
\end{thm}

Note that $r^+(P;f)\leq\ell(P;f)$ and $r^0(P;f)\leq \ell(P;\tilde f)$, where 
$\ell(P;f)$ is the number of lattice points on $\varDelta(P;f)$ minus $1$ and 
$\tilde f(x,y)=f(x,y)-f(0,0)$.

Even if $f$ does not satisfy the assumptions in Theorem~\ref{thm1.1},
by applying toric modifications successively, we can obtain an upper bound of 
the number of elements in $B_f$.
For each face $\varDelta(R_i;f)\in\Gamma_\infty^-(f)$, 
set $\mu(R_i;f)=\sum_{j=1}^\eta(\mu_j-1)$, where $\mu_1,\ldots,\mu_\eta$ are the multiplicities
of the non-zero real roots $s_1,\ldots,s_\eta$ of $g_i(v_i)=0$ in~\eqref{eq2.1} below.
Note that $\mu(R_i;f)\leq\ell(R_i;f)$.
Let $R^+$ and $R^0$ be the integers defined in Theorem~\ref{thm4.1}.
Let $|B_f|$ and $|\Sigma_f|$ denote the numbers of elements in $B_f$ and $\Sigma_f$, respectively.

\begin{thm}\label{thm5.1}
The following inequality holds:
\[
   |B_f|\leq |\Sigma_f|+\epsilon+R^0+\sum_{\varDelta(P;f)\in\Gamma_\infty^-(f)}\mu(P;f),
\]
where $\epsilon=0$ if $R^+=0$ and $\epsilon=1$ if $R^+>0$.
\end{thm}

A similar result for complex polynomial functions had been obtained in~\cite{lo}, 
see also~\cite[Corollary~6.6]{ishikawa}.

The paper is organized as follows. In Section~2, we introduce the definition of an
admissible toric compactification with respect to primitive covectors, and
give the definitions of cleaving and vanishing families and their equivalence relations.
In the subsequent three sections, 
we give the proofs of Theorem~\ref{thm1.1}, Theorem~\ref{thm4.1} and Theorem~\ref{thm5.1}.
Two examples are given in the end of Section~3. The definition of an admissible toric modification
is given in the beginning of Section~5, before giving the proof of Theorem~\ref{thm5.1}.

\section{Preliminaries}

\subsection{Toric compactification}

We first recall some definitions given in \cite{Kouchnirenko1976} which will be used in this work.
Set $f(x,y)=\sum_{(m,n)} a_{m,n}x^my^n$, where $m,n\geq 0$.
A boundary function $f_P(x,y)$ is said to be {\it non-degenerate} if the system of equations
$\frac{\partial f_P}{\partial x}=\frac{\partial f_P}{\partial y}=0$ has no solutions in $(\R\setminus\{0\})^2$.
Otherwise it is said to be {\it degenerate}.
The polynomial $f$ is called {\it convenient} if $\varDelta(f)$ intersects both positive axes.

Let $\Gamma_\infty^+(f)$ (resp. $\Gamma_\infty^0(f)$, $\Gamma_\infty^-(f)$) denote
the set of faces $\varDelta(P;f)$ of $f$ such that 
$P={}^t(p,q)$ satisfies either $p<0$ or $q<0$ and 
satisfies $d(P;f)>0$ (resp. $d(P;f)=0$, $d(P;f)<0$).
For a set $\Gamma(f)$ of faces of $\varDelta(f)$,
we say that $f$ is {\it non-degenerate} on $\Gamma(f)$
if $f_P$ is non-degenerate for any face $\varDelta(P;f)$ in $\Gamma(f)$.
Note that the non-degeneracy condition in~\cite{Kouchnirenko1976} corresponds to 
the non-degeneracy on $\Gamma_\infty^-(f)$ in this paper.

Let $f:\R^2\to\R$ be a polynomial function.
We give the definition of an admissible toric compactification
with respect to the Newton polygon $\varDelta(f)$.
Let $Q_i={}^t(p_i,q_i)$, $i=1,2,\ldots,n$, be primitive covectors which satisfy the following:
\begin{itemize}
\item[(1)] either $p_i$ or $q_i$ is negative;
\item[(2)] $\varDelta(Q_i;f)$ is a face of $\varDelta(f)$; and
\item[(3)] the indices are assigned in the counter-clockwise orientation.
\end{itemize}
Let $R_i={}^t(r_i,s_i)$, $i=1,2,\ldots,m$, be primitive covectors which satisfy the following:
\begin{itemize}
\item[(1)] $R_1={}^t(1,0)$, $R_2={}^t(0,1)$;
\item[(2)] either $r_i$ or $s_i$ is negative for each $R_i$, $i=3,\ldots,m$;
\item[(3)] $\{Q_i\}$ is contained in $\{R_3,\ldots,R_m\}$;
\item[(4)] the indices are assigned in the counter-clockwise orientation; and
\item[(5)] the determinants of the matrices $(R_i,R_{i+1})$, $i=1,\ldots,m-1$, and $(R_m,R_1)$ are $1$.
\end{itemize}
For convenience, we set $R_{m+1}=R_1$.
For each Cone$(R_i,R_{i+1})$, $i=2,\ldots,m$, an affine coordinate chart $(u_i,v_i)\in\R^2$ is
defined by the coordinate transformation
\[
x=u_i^{r_i}v_i^{r_{i+1}},\quad y=u_i^{s_i}v_i^{s_{i+1}}.
\]
Then a smooth toric variety $X$ is obtained by gluing these coordinate charts,
which is described as
\[
X=(\R^*)^2\cup\left(\bigcup_{i=1}^mE(R_i)\right)=\R^2\cup\left(\bigcup_{i=3}^mE(R_i)\right),
\]
where $E(R_i)$ is the exceptional divisor corresponding to the covector $R_i$.
The real variety $X$ is called the {\it admissible toric compactification} of $\R^2$
associated with $\{R_1,\ldots,R_m\}$.

Let $U_i$ denote the local chart with coordinates $(u_i,v_i)$ 
corresponding to Cone$(R_i,R_{i+1})$ for $i=2,\ldots,m.$
On $U_i$, the function $f$ has the form
\begin{equation}\label{eq2.1}
f(u_i,v_i)=u_i^{d(R_i;f)}v_i^{d(R_{i+1};f)}(g_i(v_i)+u_ih_i(u_i,v_i)),
\end{equation}
where $g_i$ is a polynomial of one variable $v_i$ and
$h_i$ is a polynomial of two variables $(u_i,v_i)$. The divisor $E(R_i)$ in this chart is given by $u_i=0$.

For an algebraic curve $C$ in $\R^2$, its closure in $X$ is called the {\it strict transform} of $C$.
Set $f(x,y)=x^\alpha y^\beta F(x,y)$, where $\alpha$ (resp. $\beta$) is a non-negative integer
such that $x$ (resp. $y$) does not divide $F$.
Let $V_f$, $V_F$, $V_1$ and $V_2$ denote the strict transforms of $f(x,y)=0$,
$F(x,y)=0$, $x=0$ and $y=0$ in $X$, respectively. 
In particular, $V_f=V_F\cup V_1\cup V_2$.
Note that $V_1$ (resp. $V_2$) is empty if $\alpha$ (resp. $\beta$) is $0$.

\begin{lemma}\label{lem2.1}
\begin{itemize}
\item[(1)] For $i=3,\ldots,m$, $V_F$ does not intersect $E(R_i)$ unless $R_i\in\{Q_1,\ldots,Q_n\}$.
\item[(2)] For $i=3,\ldots,m-1$, $V_F$ does not intersect $E(R_i)\cap E(R_{i+1})$.
\item[(3)] If $\alpha\geq 1$ (resp. $\beta\geq 1$) then 
$V_1$ (resp. $V_2$) intersects $E(R_m)$ (resp. $E(R_3)$) transversely.
\end{itemize}
\end{lemma}

\begin{proof}
All the assertions in this lemma are well-known.
For instance, the explanation in~\cite{oka} restricted to the two variable case 
works for real polynomial maps also.
We only check the assertion~(3) to confirm the usage of indices.
The curves $V_1$ and $E(R_m)$ are given on $U_m$ as
$\{(u_m,v_m)\in U_m\mid v_m=0\}$ and $\{(u_m,v_m)\in U_m\mid u_m=0\}$,
respectively. Hence they intersect transversely.
Similarly, $V_2$ and $E(R_3)$ are given on $U_2$ as
$\{(u_2,v_2)\in U_2\mid u_2=0\}$ and $\{(u_2,v_2)\in U_2\mid v_2=0\}$,
respectively. Hence they intersect transversely.
\end{proof}

\begin{lemma}\label{lem2.2}
Let $i$ be an index in $\{3,\ldots,m\}$. Suppose that $R_i$ satisfies one of the following:
\begin{itemize}
\item[(i)] $\varDelta(R_i;f)$ is not a bad face and $f_{R_i}$ is non-degenerate.
\item[(ii)] $\varDelta(R_i;f)$ is a bad face and $b_{R_i}(t)=0$ 
in~\eqref{eq1.1} has no non-zero real multiple root.
\end{itemize}
Then there is a one-to-one correspondence between the intersection points of $E(R_i)$ and $V_F$ 
and the non-zero real roots of $g_i(v_i)=0$. Moreover, they intersect transversely at these points.
\end{lemma}

\begin{proof}
The divisor $E(R_i)$ is given on $U_i$ as $\{(u_i,v_i)\in U_i\mid u_i=0\}$.
On the other hand, $V_F\cap U_i$ is the set 
$\{(u_i,v_i)\in U_i\mid g_i(v_i)+u_ih_i(u_i,v_i)=0\}$ 
with excluding isolated points on $u_i=0$.
Let $(0,s)$ be an intersection point of $u_i=0$ and $g_i(v_i)+u_ih_i(u_i,v_i)=0$,
where $s\in\R\setminus\{0\}$.
Since $s$ is a single root of $g_i(v_i)=0$ in both of cases (i) and (ii),
$\frac{\partial(g_i(v_i)+u_ih_i(u_i,v_i))}{\partial v_i}(0,s)\ne 0$.
Hence $g_i(v_i)+u_ih_i(u_i,v_i)=0$ is smooth at $(0,s)$, 
i.e, $(0,s)$ is not an isolated point, and $V_F$ intersects $u_i=0$ transversely
at $(0,s)$.
\end{proof}

\begin{rem}
The assumption of non-degeneracy of $f_{R_i}$ is necessary. For example
if $g_i(v_i)+u_ih_i(u_i,v_i)=(v_i-1)^2+u_i^2$ then the intersection point $(0,1)$ with $u_i=0$ is isolated.
\end{rem}

\subsection{Cleaving and vanishing at infinity}

Let $N$ be a small, compact tubular neighborhood of $\cup_{i=3}^mE(R_i)$ in $X$.

\begin{dfn}\label{dfn2.4}
A continuous family $\{(\gamma_t,\delta_t,c_t)\}_{t\in (0,1)}$ of 
triples of a proper arc $\gamma_t$ in $N\setminus\cup_{i=3}^mE(R_i)$
whose endpoints lie on the boundary $\partial N$,
a closed, connected subset $\delta_t\subset\gamma_t$, which is either a closed arc or a point, 
and a real number $c_t$ is called a {\it cleaving family} of $f$ if it satisfies the following:
\begin{itemize}
\item[(1)] $\gamma_t\subset f^{-1}(c_t)$; and
\item[(2)] $c:=\lim_{t\to\ 0}c_t$ satisfies $|c|<\infty$ 
and $\delta:=\lim_{t\to 0}\delta_t\subset \cup_{i=3}^mE(R_i)$.
\end{itemize}
If there exists a cleaving family with limit $f=c$, 
then we say that the curve $f=c$ is {\it cleaving at infinity}.
\end{dfn}

Note that the definition of a cleaving family depends on 
the compactification $X$ of $\R^2$, though 
the existence of a cleaving family and its value $c$ do not.
In this sense, the statement ``$f=c$ is cleaving at infinity'' does not depend on the choice of $X$.
This definition coincides with that in~\cite[p.30]{cp} if we state it without compactification.
Obviously, if $f=c$ is cleaving at infinity then $c\in B_f$.

\begin{dfn}\label{dfn2.5}
A continuous family $\{(C_t, c_t)\}_{t\in (0,1)}$ of 
pairs of a real number $c_t$ and a connected component $C_t$ of $f=c_t$ in $\R^2$ 
is called a {\it vanishing family} if it satisfies the following:
\begin{itemize}
\item[(1)] $C_t\subset N\setminus\cup_{i=3}^mE(R_i)$; and
\item[(2)] $c:=\lim_{t\to\ 0}c_t$ satisfies $|c|<\infty$ 
and $C:=\lim_{t\to 0}C_t\subset \cup_{i=3}^mE(R_i)$.
\end{itemize}
If there exists a vanishing family with limit $f=c$, then
we say that the curve $f=c$ is {\it vanishing at infinity}.
\end{dfn}

Note that the definition of a vanishing family
depends on the compactification $X$ of $\R^2$, though 
the existence of a vanishing family and its value $c$ do not.
In this sense, the statement ``$f=c$ is vanishing at infinity'' 
does not depend on the choice of $X$. 
In~\cite{tz}, the value $c\in B_f$ is characterized by the first betti numbers and Euler characteristics
of fibers and ``vanishing'' and ``splitting'' phenomena.
The definition of a vanishing family coincides with the ``vanishing'' in~\cite{tz}
if we state it without compactification.
Obviously, if $f=c$ is vanishing at infinity then $c\in B_f$.

\begin{lemma}[\cite{tz,cp}, see p.31 in~\cite{cp}]\label{lem2.6}
Suppose that $c\in B_f$. Then one of the following holds:
\begin{itemize}
\item[(i)] $c\in \Sigma_f$;
\item[(ii)] $f=c$ is cleaving at infinity;
\item[(iii)] $f=c$ is vanishing at infinity.
\end{itemize}
\end{lemma}

Since the definitions of these families depend on the choice of the compact neighborhoods $N$, 
the parameter $t$ and the subsets $\{\delta_t\}_{t\in(0,1)}$, we need to introduce an equivalence relation
to remove these ambiguities. The equivalence relation is defined as follows.

\begin{dfn}\label{dfn2.7}
\begin{itemize}
\item[(1)] 
Two cleaving families $\{(\gamma_t,\delta_t,c_t)\}_{t\in (0,1)}$ 
and $\{(\gamma'_t,\delta'_t,c'_t)\}_{t'\in (0,1)}$, defined in compact tubular neighborhoods $N$ 
and $N'$ of $\cup_{i=3}^mE(R_i)$ respectively,
are {\it equivalent} if there exists $\ve>0$ such that for any $s\in (0,\ve)$
there exists $s'\in (0,1)$ such that $c_s=c'_{s'}$ and
$\gamma_s\cap \gamma'_{s'}\ne\emptyset$.
\item[(2)]
Two vanishing families $\{(C_t,c_t)\}_{t\in (0,1)}$ 
and $\{(C'_t,c'_t)\}_{t'\in (0,1)}$, defined in compact tubular neighborhoods $N$ 
and $N'$ of $\cup_{i=3}^mE(R_i)$ respectively,
are {\it equivalent} if there exists $\ve>0$ such that for any $s\in (0,\ve)$
there exists $s'\in (0,1)$ such that $c_s=c'_{s'}$ and $C_s=C'_{s'}$.
\end{itemize}
\end{dfn}

Later, we will count the numbers of cleaving and vanishing families
up to these equivalence relations.

\section{Proof of Theorem~\ref{thm1.1} and examples}

Theorem~\ref{thm1.1} will follow from the next proposition.

\begin{prop}\label{prop3.1}
Suppose that $f(x,y)$ is non-degenerate 
on $\Gamma_\infty^+(f)\cup\Gamma_\infty^-(f)$ and that
$b_P(t)=0$ in~\eqref{eq1.1} 
has no non-zero real multiple root for any bad face $\varDelta(P;f)$.
Then, $0\in B_f$ if and only if either {\rm (i)}~$0\in \Sigma_f$ or
{\rm (ii)}~there exists $\varDelta(P;f)\in\Gamma_\infty^+(f)$
such that $f_P(x,y)=0$ has a solution in $(\R\setminus\{0\})^2$.
Moreover, if it is in case~{\rm (ii)} then $f=0$ is cleaving at infinity.
\end{prop}

We divide the proof into three lemmas. 
Note that the proofs of the first two lemmas for complex polynomials
are written, for example, in~\cite{nz, son}, which are based on the Curve Selection Lemma at infinity,
and their arguments work in real case also.
We here give different proofs based on the toric compactification $X$.
Let $N$ denote a small, compact tubular neighborhood of $\cup_{i=3}^mE(R_i)$ in $X$.

\begin{lemma}\label{lem3.2}
Suppose that $f$ is convenient and non-degenerate on $\Gamma_\infty^-(f)$.
Then $0\in B_f$ if and only if $0\in \Sigma_f$.
\end{lemma}

\begin{proof}
Since $\Sigma_f\subset B_f$, it is enough to show that if $0\in B_f$ then $0\in \Sigma_f$. 
Assume that $0\not\in \Sigma_f$. 
By Lemma~\ref{lem2.6}, it is enough to check that $f=0$ is 
not cleaving and not vanishing at infinity.
Note that there is no bad face since $f$ is convenient, and
$V_f$ intersects $\cup_{i=3}^m E(R_i)$ transversely by Lemma~\ref{lem2.1} and ~\ref{lem2.2}.

We first prove that $f=0$ is not cleaving at infinity. 
Assume that $f=0$ is cleaving at infinity. 
Then $V_f$ must intersect $E(R_i)$ for some $i=3,\ldots,m$. 
Let $p$ be an intersection point of $V_f$ and $E(R_i)$.
Note that $p\in\delta$, where $\delta$ is the limit of closed, connected sets $\{\delta_t\}_{t\in (0,1)}$ 
in Definition~\ref{dfn2.4}.
Since $f$ is convenient, we have $d(R_i;f)<0$. 
Then any nearby fiber of $V_f$ in $N$ near $p$ is a simple arc connecting 
a point near $V_f\cap\partial N$ and the point $p$, see Figure~\ref{fig1}.
Hence the fibration at infinity near $p$ is trivial,
which contradicts the assumption that $f=0$ is cleaving at $p$.
\begin{figure}[htbp]
\centerline{\input{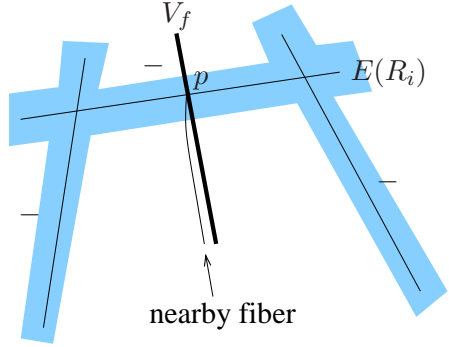}}
  \caption{The triviality of the fibration at infinity in the case $d(R_i;f)<0$. 
The sign $-$ means that $E(R_i)$ satisfies $d(R_i;f)<0$. The painted region is the neighborhood $N$.
\label{fig1}}
\end{figure}

Next we check that $f=0$ is not vanishing at infinity.
Since $d(R_i;f)<0$ for any $i=3,\ldots,m$, if a vanishing family exists then there exists a sequence $\{(t_j,s)\}_{j\in\N}$ of points on $U_i$
for some $i\in\{3,\ldots,m\}$ such that 
$\lim_{j\to\infty}t_j=0$, $\lim_{j\to\infty}f|_{U_i}(t_j,s)=0$ 
and $g_i(s)\ne 0$. However, from~\eqref{eq2.1}, 
we see that $|\lim_{j\to\infty}f|_{U_i}(t_j,s)|=\infty$,
which is a contradiction. Thus $0\not\in B_f$ by Lemma~\ref{lem2.6}.
\end{proof}

\begin{lemma}\label{lem3.3}
Suppose that $f(0,0)\ne 0$ and $f$ is not convenient.
Suppose further that $f$ is non-degenerate on $\Gamma_\infty^-(f)$
and $b_P(t)=0$ in~\eqref{eq1.1} 
has no non-zero real multiple root for any bad face $\varDelta(P;f)$.
Then $0\in B_f$ if and only if $0\in \Sigma_f$.
\end{lemma}

\begin{proof}
We assume $0\not\in\Sigma_f$ and prove $0\not\in B_f$.
By Lemma~\ref{lem2.6}, it is enough to check that $f=0$ is not cleaving and not vanishing at infinity.

First we show that $f=0$ is not cleaving at infinity.
Let $R_i$ be a covector such that $E(R_i)$ intersects $V_f$, where $i=3,\ldots,m$.
If $d(R_i;f)=0$ then the condition (ii) in Lemma~\ref{lem2.2} holds by the assumption.
By Lemma~\ref{lem2.2}, all nearby fibers of $f=0$ near $E(R_i)$ 
are transverse to $E(R_i)$, see Figure~\ref{fig2}.
Hence the fibration at the infinity is trivial.
The triviality also holds if $d(R_i;f)<0$ as we had seen in the proof of Lemma~\ref{lem3.2}.

\begin{figure}[htbp]
    \centerline{\input{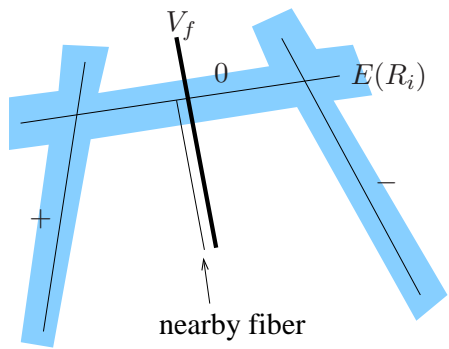}}
  \caption{The triviality of the fibration at infinity in the case $d(R_i;f)=0$.
The sign $0$ means that $E(R_i)$ satisfies $d(R_i;f)=0$.\label{fig2}}
\end{figure}

Next we check that $f=0$ is not vanishing at infinity.
If there is a vanishing family whose limit intersects $E(R_i)$ with $d(R_i;f)=0$
then, since $f(u_i,v_i)$ in~\eqref{eq2.1} has no factor $u_i^{-1}$, 
the limit in $U_i$ should be given by $f(u_i,v_i)=0$,
which this is nothing but $V_f\cap U_i$.
If $g_i(v_i)=0$ in ~\eqref{eq2.1} has a non-zero real solution then,
since it is not a multiple root, the limit cannot be contained in $E(R_i)$ with $d(R_i;f)=0$.
If $g_i(v_i)=0$ has no non-zero real solution then $V_f\cap E(R_i)\cap U_i=\emptyset$.
Hence, in either case, there is no vanishing family.
The limit cannot intersect $E(R_i)$ with $d(R_i;f)<0$ 
by the same reason as we had seen in the proof of Lemma~\ref{lem3.2}. This completes the proof.
\end{proof}

Finally, we study the case where either $x|f$ or $y|f$.

\begin{lemma}\label{lem3.4}
Suppose that $f$ is non-degenerate on $\Gamma_\infty^+(f) \cup\Gamma_\infty^-(f)$
and that $b_P(t)=0$ in~\eqref{eq1.1} has no non-zero real multiple root for any bad face $\varDelta(P;f)$.
Suppose further that either $x|f$ or $y|f$.
Then, $0\in B_f$ if and only if either {\rm (i)}~$0\in \Sigma_f$ or 
{\rm (ii)}~there exists $\varDelta(P;f)\in\Gamma_\infty^+(f)$
such that $f_P(x,y)=0$ has a solution in $(\R\setminus\{0\})^2$.
Moreover, if it is in case~{\rm (ii)} then $f=0$ is cleaving at infinity.
\end{lemma}

\begin{proof}
Set $f(x,y)=x^\alpha y^\beta F(x,y)$ with either $\alpha>0$ or $\beta>0$.
For convenience, we choose $R_3,\ldots,R_m$ such that $d(R_m;f)>0$ and $f_{R_m}$ is a monomial if $\alpha>0$ and 
that $d(R_3;f)>0$ and $f_{R_3}$ is a monomial if $\beta>0$.
Since the indices of the covectors $\{R_1,\ldots,R_m\}$ are assigned in
the counter-clockwise orientation, if $\alpha>0$ then there exists an index $k$ such that
$d(R_k;f)\leq 0$ and $d(R_i;f)>0$ for $i>k$.
Similarly, if $\beta>0$ then there exists an index $k'$ such that 
$d(R_i;f)>0$ for $3\leq i\leq k'$ and $d(R_{k'+1};f)\leq 0$. Note that $E(R_k)\cap E(R_{k'})=\emptyset$
when $\alpha>0$ and $\beta>0$.


Suppose that there exists $\varDelta(R_i;f)\in\Gamma_\infty^+(f)$ with $i\geq k$
such that $f_{R_i}(x,y)=0$ has a solution in $(\R\setminus\{0\})^2$.
Let $i_0$ be the largest index such that $d(R_{i_0};f)>0$ and $f_{R_i}(x,y)$ has a solution in $(\R\setminus\{0\})^2$.
We consider the real toric variety $X$ obtained by the admissible toric compactification
of $\R^2$ associated with $\{R_1,\ldots,R_m\}$.
Let $\gamma$ be a branch of $V_f$ in $N$ intersecting $E(R_{i_0})$
and being nearest to $E(R_{i_0+1})$.
By Lemma~\ref{lem2.1}~(3), 
$V_1\cap N$ is a short arc in $N$ intersecting $E(R_m)$ transversely, see Figure~\ref{fig3}.
Thus we can find a cleaving family between $V_1$ and $\gamma$ in $N$.
If there exists $\varDelta(R_i;f)\in\Gamma_\infty^+(f)$ with $3\leq i\leq k'$ 
such that $f_{R_i}(x,y)=0$ has a solution in $(\R\setminus\{0\})^2$
then there exists a cleaving family by the same reason.
Thus, in either case, we have $0\in B_f$.

\begin{figure}[htbp]
    \centerline{\input{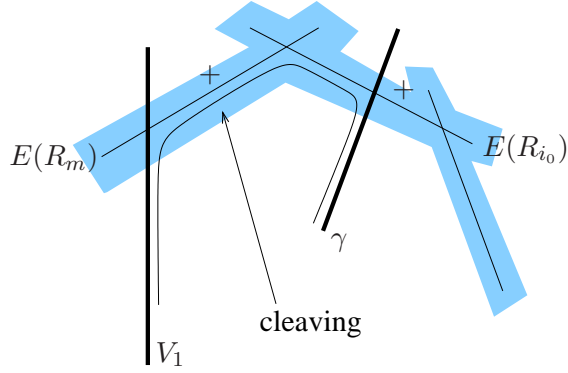}}
  \caption{A cleaving family between $V_1$ and $\gamma$.
The sign $+$ means that $E(R_i)$ satisfies $d(R_i;f)>0$.
\label{fig3}}
\end{figure}

Now we prove the converse.
Suppose that there does not exist $\varDelta(P;f)\in\Gamma_\infty^+(f)$
such that $f_P(x,y)=0$ has a solution in $(\R\setminus\{0\})^2$.
By Lemma~\ref{lem2.6}, it is enough to check that $f=0$ is not cleaving and not vanishing at infinity.
We first check that $f=0$ is not cleaving.
For any $\varDelta(R_i;f)\in \Gamma_\infty^+(f)$,
since $f_{R_i}(x,y)=0$ has no solution in $(\R\setminus\{0\})^2$,
$V_F$ does not intersect $E(R_i)$ by Lemma~\ref{lem2.2}.
If $V_F$ intersects $E(R_i)$ with $d(R_i;f)=0$ then,
by the same argument as in the proof of Lemma~\ref{lem3.3}, the fibration at infinity near $E(R_i)$ 
is trivial, see Figure~\ref{fig2}.
The triviality also holds in the case where $V_F$ intersects $E(R_i)$ with $d(R_i;f)<0$
as we have seen in the proof of Lemma~\ref{lem3.2}, see Figure~\ref{fig1}.
This shows that there is no cleaving family near the intersection of $V_F$ with $\cup_{i=3}^mE(R_i)$.

By Lemma~\ref{lem2.1}, it remains to show that
there is no cleaving family near the intersection of $V_1$ with $E(R_m)$
and near the intersection of $V_2$ with $E(R_3)$.
We only check the former case. The latter case is proved similarly.
On $U_m$, $f$ has the form
\[
f(u_m,v_m)=u_m^{d(R_m;f)}v_m^{d(R_1;f)}(g_m(v_m)+u_mh_m(u_m,v_m)),
\]
where $V_f\cap U_m$ corresponds to the curve $v_m^{d(R_1;f)}(g_m(v_m)+u_mh_m(u_m,v_m))=0$.
Since the covectors $R_3,\ldots,R_m$ are chosen such that $f_{R_m}$ is a monomial, 
$V_f$ intersects $E(R_m)$ only at $E(R_m)\cap V_1$ by Lemma~\ref{lem2.1}.
If $\Gamma_\infty^0(f)\ne \emptyset$ then a nearby fiber of $f$ passing near $V_1$
intersects $\cup_{i=3}^mE(R_i)$ at a point near 
$E(R_{k-1})\cap E(R_{k})$ as shown on the left in Figure~\ref{fig4}. 
Therefore, the fibration is trivial at infinity.
If $\Gamma_\infty^0(f)=\emptyset$ then nearby fibers intersect $E(R_{k-1})\cap E(R_{k})$ 
as shown on the right in Figure~\ref{fig4}
since $\varDelta(R_{k-1};f)\in\Gamma_\infty^-(f)$ and $\varDelta(R_{k};f)\in\Gamma_\infty^+(f)$. 
Therefore, the fibration is again trivial at infinity.
Thus $f=0$ is not cleaving at infinity.

\begin{figure}[htbp]
   \centerline{\input{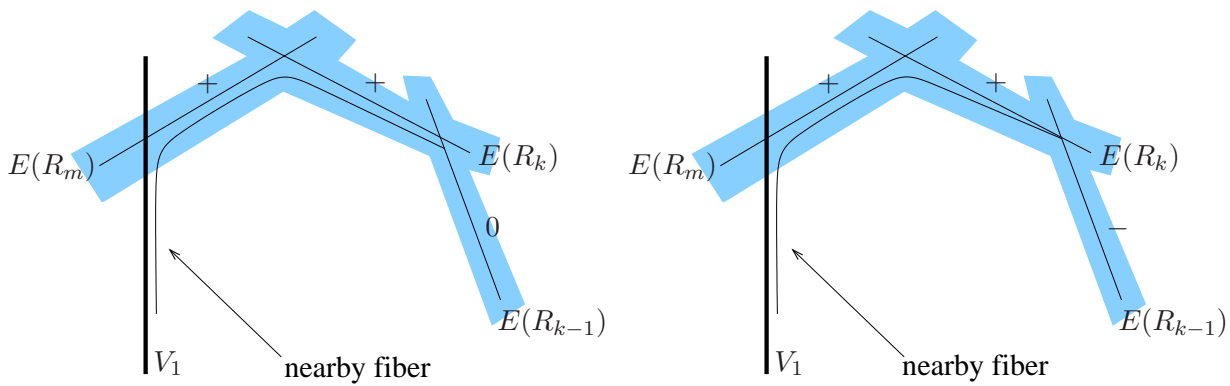}}
  \caption{The triviality of the fibration at infinity for nearby fibers passing near $V_1$.\label{fig4}}
\end{figure}

Next we check that $f=0$ is not vanishing at infinity.
As we explained in the proof of Lemma~\ref{lem3.3}, a vanishing family does not exist
in a neighborhood of $E(R_i)$ with $d(R_i;f)=0$.
It does not exist near $E(R_i)$ with $d(R_i;f)>0$ also since a nearby fiber cannot stay in $N$
as we had seen in Figure~\ref{fig4}. 
A vanishing family does not exist near $E(R_i)$ with $d(R_i;f)<0$
by the same reason as we had seen in the proof of Lemma~\ref{lem3.2}. 
This completes the proof.
\end{proof}

\begin{proof}[Proof of Proposition~\ref{prop3.1}]
The assertion follows from Lemmas~\ref{lem3.2}, \ref{lem3.3} and~\ref{lem3.4}.
\end{proof}

\begin{proof}[Proof of Theorem~\ref{thm1.1}]
If it is in case (ii), by applying Proposition~\ref{prop3.1} to 
$\tilde f(x,y)=f(x,y)-c$, we have $c\in B_f$.

If it is in case (iii) then $f-c$ has the form 
\[
   f(u_i,v_i)-c=v_i^{d(R_{i+1};f)}(\tilde g_i(v_i)+u_ih_i(u_i,v_i)),
\]
where $\tilde g_i(v_i)=g_i(v_i)-cv_i^{-d(R_{i+1};f)}$, and 
there exists a non-zero real root $s$ of $\tilde g_i(v_i)=0$ with multiplicity $2$.
The strict transform $V_{f-c}$ of $f-c=0$ intersects $E(R_i)$ at $(u_i,v_i)=(0,s)$ 
with multiplicity $2$.
Let $U$ be a small neighborhood of $(0,s)$ in $N$
such that $U\setminus E(R_i)$ consists of two connected components, say $U'$ and $U''$.
There are three cases: (1) $V_{f-c}$ intersects both of $U'$ and $U''$;
(2) $V_{f-c}$ intersects one of them and does not intersect the other;
(3) $V_{f-c}$ does not intersect both of $U'$ and $U''$.
In case~(2), there is a cleaving family and a vanishing family as shown in Figure~\ref{fig5}.
Thus we have $c\in B_f$.
\begin{figure}[htbp]
   \centerline{\input{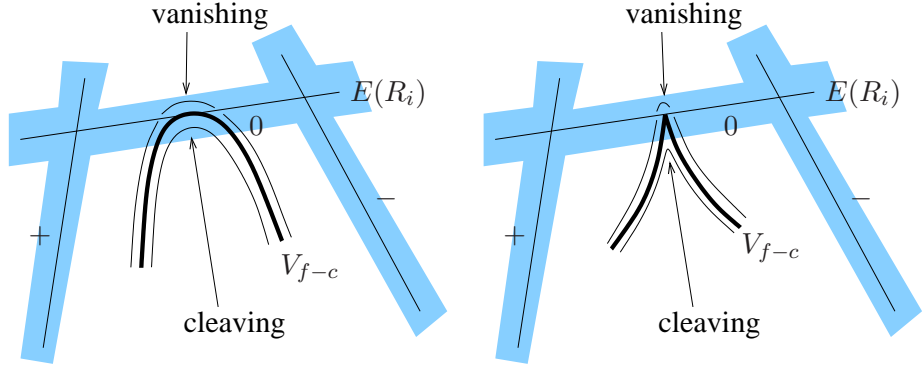}}
  \caption{Cleaving and vanishing at infinity on $E(R_i)$ with $\varDelta(R_i;f)\in\Gamma_\infty^0(f)$ (Case where $V_{f-c}$ is in one side).\label{fig5}}
\end{figure}

In case~(1), since the multiplicity is $2$, 
$V_{f-c}$ in $U$ consists of either two curves
intersecting each other transversely and also intersecting $E(R_i)$ transversely,
see on the left in Figure~\ref{fig6}, or one curve with multiplicity $2$
intersecting $E(R_i)$ transversely.
In the former case, $f=c$ is cleaving at infinity from both sides as shown in the figure.
In the latter case, $c\in\Sigma_f$.
In case~(3), $f=c$ has vanishing families from both sides, see on the right in Figure~\ref{fig6}.
In any case, we have $c\in B_f$.
\begin{figure}[htbp]
   \centerline{\input{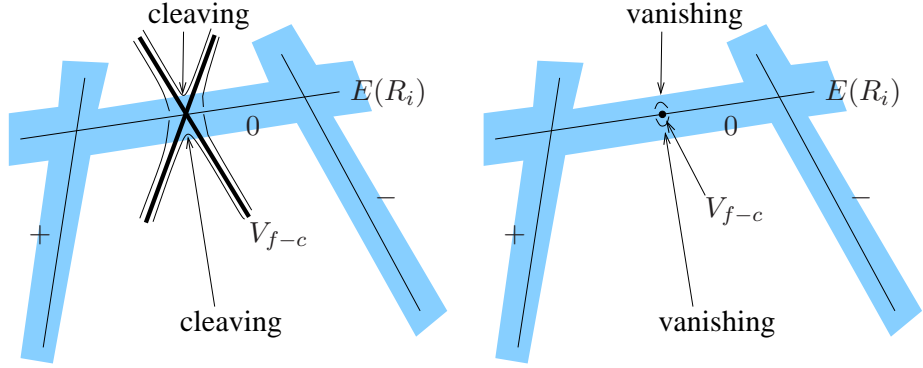}}
  \caption{Cleaving and vanishing at infinity on $E(R_i)$ with $\varDelta(R_i;f)\in\Gamma_\infty^0(f)$ (Case where $V_{f-c}$ is either in both sides or isolated).\label{fig6}}
\end{figure}

Conversely, if both of (ii) and (iii) are not satisfied then, applying 
Proposition~\ref{prop3.1} to $f(x,y)-c$,
we can conclude that $c\not\in B_f$ unless $c\in \Sigma_f$.
\end{proof}

\begin{rem}
The ``Morse condition'' on bad faces is cruciel especially in case~(1) in the above proof.
If the multiplicity $\mu$ is odd then $V_{f-c}$ can be one curve being tangent to $E(R_i)$ 
with multiplicity $\mu$ and intersects both of $U'$ and $U''$.
In this case, the fibration is trivial in this neighborhood. 
Hence we cannot generalize the assertion in the case where $\mu$ is odd.

If the multiplicity $\mu$ is even, $c\not\in\Sigma_f$ and it is in case~(1) then
there are at least two branches of $V_{f-c}$ passing though the intersection point $(0,s)$ 
and thus a cleaving family exists.
Therefore the assertion in Theorem~\ref{thm1.1} holds even if we replace
the ``Morse condition'' on bad faces into the ``even multiplicities''.
\end{rem}

\begin{ex}
Consider the polynomial function $f(x,y)=x(1+x^my^{2n})$, where $m,n\geq 1$.
From the Newton polygon $\varDelta(f)$, 
there are two covectors orthogonal to the face of $\varDelta(f)$, i.e.,
\[
Q_1=\begin{pmatrix} -2n \\ m \end{pmatrix},\;\,
Q_2=\begin{pmatrix} 2n \\ -m \end{pmatrix}.
\]
We can easily check that the conditions~(i) and~(iii) in Theorem~\ref{thm1.1} are not
satisfied. If $c\ne 0=f(0,0)$ then (ii) is also not satisfied.
When $c=0$, only the covector $Q_2$ satisfies $\varDelta(Q_2;\tilde f)\in\Gamma_\infty^+(\tilde f)$. 
We can easily check that (ii) is satisfied if and only if $m$ is odd. 
Thus $B_f=\{0\}$ if $m$ is odd and $B_f=\emptyset$ if $m$ is even.

We here explain how the cleaving and vanishing families appear in a real toric variety
in the case where $f(x,y)=x(1+xy^2)$.
Set 
\[
\begin{split}
&R_1=\begin{pmatrix} 1 \\ 0 \end{pmatrix},\;\,
R_2=\begin{pmatrix} 0 \\ 1 \end{pmatrix},\;\,
R_3=\begin{pmatrix} -1 \\ 1 \end{pmatrix},\;\,
R_4=Q_1=\begin{pmatrix} -2 \\ 1 \end{pmatrix},\\
&R_5=\begin{pmatrix} -1 \\ 0 \end{pmatrix},\;\,
R_6=Q_2=\begin{pmatrix} 2 \\ -1 \end{pmatrix},\;\,
R_7=R_1.
\end{split}
\]
These primitive covectors satisfy the conditions in Section~2 and 
the associated admissible toric compactification $X$ becomes as shown in Figure~\ref{fig7}.
We can see from the figure that 
there are two cleaving families up to equivalence relation defined in Definition~\ref{dfn2.7} 
and there is no vanishing family.
\end{ex}
\begin{figure}[htbp]
   \centerline{\input{fig8.pstex_t}}
  \caption{A connected component of $f=\ve$ and a connected component of $f=-\ve$,
with sufficiently small $\ve>0$, are described. 
Both of them are cleaving as $\ve\to 0$.\label{fig7}}
\end{figure}

\begin{ex}
Consider the polynomial function
\[
   f(x,y)=x+\frac{1}{m}x^my^m+\frac{2a}{m+1}x^{m+1}y^{m+1}+\frac{1}{m+2}x^{m+2}y^{m+2}
\]
with $m\geq 2$.
It has no singular point and hence $\Sigma_f=\emptyset$.

From the Newton polygon $\varDelta(f)$, the covectors orthogonal to the faces are
\[
Q_1=\begin{pmatrix} -m-2 \\ m+1 \end{pmatrix},\;\,
Q_2=\begin{pmatrix} 1 \\ -1 \end{pmatrix},\;\,
Q_3=\begin{pmatrix} m \\ 1-m \end{pmatrix}.
\]
Only the face $\varDelta(Q_2;f)$ is a bad face.
To apply Theorem~\ref{thm1.1}, we need to assume that 
$b_{Q_2}(t)=t^m(\frac{1}{m}+\frac{2a}{m+1}t+\frac{1}{m+2}t^2)$
is a Morse function on $\R\setminus\{0\}$. The critical points are the roots of
$\frac{db_{Q_2}}{dt}(t)=t^{m-1}(1+2at+t^2)=0$.
Thus $b_{Q_2}$ is Morse on $\R\setminus\{0\}$ if and only if $a\ne \pm 1$.
If $-1<a<1$ then $b_{Q_2}$ has no critical point on $\R\setminus\{0\}$.
If $|a|>1$ then 
$t_0=-a-\sqrt{a^2-1}$ and $t_1=-a+\sqrt{a^2-1}$ are the critical points of $b_{Q_2}$.
The face $\varDelta(Q_3;f)$ is in $\Gamma_\infty^+(f)$
and $f_{Q_3}(x,y)=x+x^my^m=0$ has a solution $(-1,-1)$ in $(\R\setminus\{0\})^2$.
Hence $0\in B_f$.
The bifurcation set $B_f$ is now determined for $|a|\ne 1$: 
$B_f=\{0, b_Q(t_0), b_Q(t_1)\}$ if $|a|>1$ and $B_f=\{0\}$ if $|a|<1$.

Now we explain how the cleaving and vanishing families appear in a real toric variety.
Set 
\[
\begin{split}
&R_1=\begin{pmatrix} 1 \\ 0 \end{pmatrix},\;\,
R_2=\begin{pmatrix} 0 \\ 1 \end{pmatrix},\;\,
R_3=\begin{pmatrix} -1 \\ 1 \end{pmatrix},\;\,
R_4=Q_1=\begin{pmatrix} -m-2 \\ m+1 \end{pmatrix},\\
&R_5=Q_2=\begin{pmatrix} 1 \\ -1 \end{pmatrix},\;\,
R_6=Q_3=\begin{pmatrix} m \\ 1-m \end{pmatrix},\\
&R_{6+k}=\begin{pmatrix} m-k \\ 1-m+k \end{pmatrix} (k=1,\ldots,m-3),\;\,
R_{m+4}=\begin{pmatrix} 2 \\ -1 \end{pmatrix},\;\,
R_{m+5}=R_1.
\end{split}
\]
These primitive covectors satisfy the conditions in Section~2 and 
the associated admissible toric compactification $X$ becomes as shown in Figure~\ref{fig8},
which is in the case where $m=8$ and $|a|>1$.
If $|a|<1$ then $V_f$ does not intersect $E(R_5)$.

\begin{figure}[htbp]
   \centerline{\input{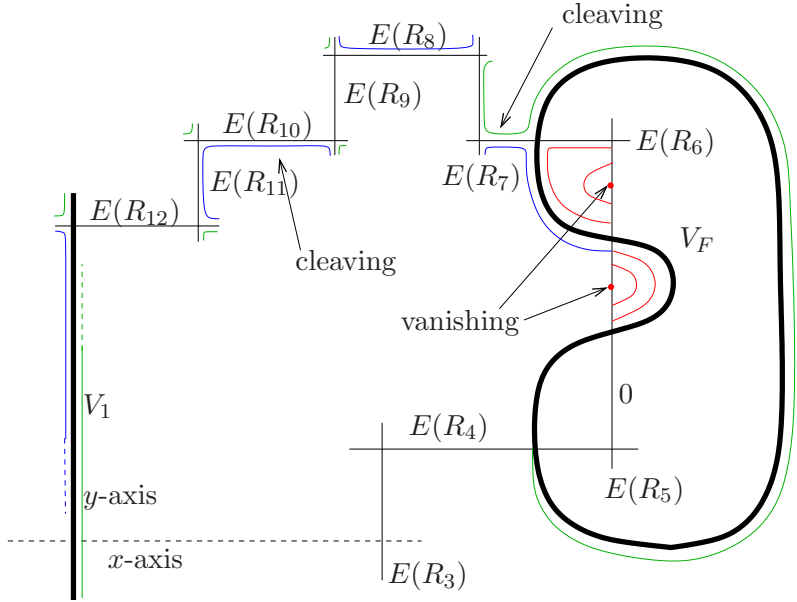}}
  \caption{A part of a connected component of $f=\ve$ and a part of a connected component of $f=-\ve$,
with sufficiently small $\ve>0$, are described. 
Both of them are cleaving as $\ve\to 0$.\label{fig8}}
\end{figure}

On the local chart $U_5$ with coordinates $(u_5,v_5)$, for each $j=0,1$, 
we have
\[
   f(u_5,v_5)-f(t_j)=(v_5-t_j)^2\hat g_j(v_5)+u_5v_5^8
\]
with $\hat g_j(t_j)\ne 0$.
Thus $V_{f-f(t_j)}$ is tangent to $E(R_5)$ at $(u_5,v_5)=(0,t_j)$ with multiplicity $2$.
This is in case~(2) in the proof of Theorem~\ref{thm1.1}. 
Hence we see that there are a cleaving family and a vanishing family for each $j=0,1$.
Since a vanishing family does not appear in the settings in Lemma~\ref{lem3.3} and~\ref{lem3.4},
we see that there is no other vanishing family.
There are two cleaving families with limit $f=0$ as shown in Figure~\ref{fig8}.
Here we count the numbers of cleaving and vanishing families 
up to equivalence relations in Definition~\ref{dfn2.7}.
In summary, this example has four cleaving families and two vanishing families.
For other $m$'s more than $1$, we can easily check that $f$ also has the same numbers of cleaving 
and vanishing families.
\end{ex}

\section{Proof of Theorem~\ref{thm4.1}}

In this section, we give the proof of Theorem~\ref{thm4.1},
which determines the number of cleaving and vanishing families
counted up to equivalence relations in Definition~\ref{dfn2.7}.

\begin{proof}[Proof of Theorem~\ref{thm4.1}]
Let $E^+$ and $E^0$ denote the union of $E(R_i)$'s with $\varDelta(R_i;f)\in\Gamma_\infty^+(f)$
and with $\varDelta(R_i;f)\in\Gamma_\infty^0(f)$, 
and let $N^+$ and $N^0$ denote a small, compact neighborhood of $E^+$ and $E^0$ in $X$, respectively.
A cleaving family appears either in $N^+$ or $N^0$, see Lemma~\ref{lem3.4} and
the proof of Theorem~\ref{thm1.1}.
A vanishing family appears only in $N^0$, see the proof of Theorem~\ref{thm1.1}.

First we observe it in $N^+$. We may assume $f(0,0)=0$ by replacing $f$ by $f(x,y)-f(0,0)$.
Suppose that $E^+\ne\emptyset$.
We may assume that $f$ has the form $f(x,y)=x^\alpha y^\beta F(x,y)$ with either $\alpha>0$ or $\beta>0$.
We set $E^+=E_x^+\cup E_y^+$, where $E^+_x=\cup_{i=k}^mE(R_i)$ and $E^+_y=\cup_{i=3}^{k'}E(R_i)$.
Here $k$ is the index such that $d(R_{k-1};f)\leq 0$ and $d(R_k;f)>0$ (cf. Figure~\ref{fig4})
and $k'$ is the index such that $d(R_{k'};f)>0$ and $d(R_{k'+1};f)\leq 0$.
Note that if $\alpha=0$ (resp. $\beta=0$) then $E_x^+$ (resp. $E_y^+$) is empty and that $E_x^+\cap E_y^+=\emptyset$.
Let $R_x^+$ (resp. $R_y^+$) be 
the sum of $r^+(R_i;f)$'s for $i\geq k$ (resp. $3\leq i\leq k'$) with $\varDelta(R_i;f)\in\Gamma^+_\infty(f)$.
Note that $R^+=R_x^++R_y^+$.

We first count the number of cleaving families in a compact neighborhood $N^+_x$ of $E^+_x$.
Remark that $R^+_x$ is equal to the number of intersection points $V_F\cap E^+_x$. 
A nearby fiber in $N^+_x$ yields
a ``cleaving'' if and only if both of the endpoints of the fiber in $N^+_x$ is on the boundary $\partial N^+_x$.
If an endpoint is not on $\partial N^+_x$ then it is on $E(R_{k-1})$ if $d(R_{k-1};f)=0$ and 
on the intersection $E(R_{k-1})\cap E(R_k)$ if $d(R_{k-1};f)<0$.
Such an endpoint can appear on all of the four quadrants on the chart $U_{k-1}$ 
corresponding to Cone$(R_{k-1},R_k)$.
Since $V_1$ intersects $E(R_m)$ at one point, $V_f$ intersects $E_x^+$ at $R^+_x+1$ points.
Adding the $4$ endpoints on $U_{k-1}$, there are totally $4R^+_x+8$ endpoints. 
Hence the curve $\{f=\ve\}\cup\{f=-\ve\}$ has $2R^+_x+4$ connected components in $N^+_x$.
However, the endpoints lying on $E(R_{k-1})$ do not contribute to ``cleavings''.
Moreover, there is no connected component both of whose endpoints are on $E(R_{k-1})$.
This can be checked as follows: Try to describe a curve in $N^+_x$ starting at one of the endpoints
on $E(R_{k-1})$. Then we meet $V_f$ before coming back near $E(R_{k-1})$. 
Thus the curve must go out from $\partial N^+_x$.
There are four connected components which do not contribute to ``cleavings''.
Hence the number of cleaving families in $N^+_x$ is $2R^+_x$.
This is true even if $E^+_x=\emptyset$ since there is no cleaving family in this case.

The number of cleaving families in a compact neighborhood $N^+_y$ of $E^+_y$
can be counted by the same way and it becomes $2R^+_y$.
Since $E_x^+\cap E_y^+=\emptyset$, the countings in $N^+_x$ and $N^+_y$ do not conflict.
Hence the total number of cleaving families in $N^+$ is $2R^+_x+2R^+_y=2R^+$.

Next we observe it in $N^0$.
Since $f$ has only isolated singularities, each critical point of $b_P(t)$ corresponds to
an intersection point of $V_F$ and $E^0$ as shown in Figures~\ref{fig5} and~\ref{fig6}.
In either case,  for each intersection point,
the sum of the number of cleaving families and that of vanishing families is $2$.
Hence the total number of cleaving and vanishing families in $N^0$ is $2R^0$.
This completes the proof.
\end{proof}

\section{An upperbound of $|B_f|$}

In this section, we do not assume that $f$ is non-degenerate on $\Gamma_\infty^+(f)\cup\Gamma_\infty^-(f)$
and also do not assume that $f_P$ is Morse on a bad face $\varDelta(P;f)\in\Gamma_\infty^0(f)$.

We will prove Theorem~\ref{thm5.1} by applying successive admissible toric modifications for 
each singularity on $\cup_{i=3}^mE(R_i)$ appearing due to degeneracies.
We first introduce an admissible toric modification.
Though an admissible toric modification is usually defined for a polynomial function or 
a locally analytic function, we define it for rational functions given as in~\eqref{eq2.1}.
Note that such a modification had been used in~\cite{ishikawa} 
for studying singularities at infinity of complex polynomial functions.

Let $U\subset \R^2$ be a small neighborhood of the origin and let $\hat f:U\to\R$ be
a real rational function on $U$ whose expansion is given by 
$\hat f(x,y)=\sum_{(m,n)}a_{m,n}x^my^n$, where $(m,n)\in\Z$ with $m>-M$ for some non-negative
integer $M$ and $n\geq 0$.
We define the Newton polygon $\varDelta^{\text{\rm loc}}(\hat f)$ of $\hat f$ 
by the convex hull of $\cup_{(m,n)}((m,n)+\R^2_{\geq 0})$,
where $\R_{\geq 0}=\{x\in\R\mid x\geq 0\}$ and
the union is taken for all $(m,n)$ such that $a_{m,n}\ne 0$.
For a given primitive covector $P={}^t(p,q)$ with $p,q>0$,
let $d(P;\hat f)$ denote the minimal value of the linear function $pX+qY$, 
where $(X,Y)\in\varDelta^{\text{\rm loc}}(\hat f)$. 
Set $\varDelta(P;\hat f):=\{(X,Y)\in\varDelta^{\text{\rm loc}}(\hat f)\mid pX+qY=d(P;\hat f)\}$,
which is called a {\it face} if $\dim \varDelta(P;\hat f)=1$.
The partial sum $\hat f_P(x,y)=\sum_{(m,n)\in\varDelta(P;\hat f)}a_{m,n}x^my^n$
is called the {\it boundary function} for the covector $P$. 
If $\varDelta(P;\hat f)$ is a face then it is called the {\it face function}.
A boundary function $\hat f_P$ is said to be {\it degenerate} if
$\frac{\partial \hat f_P}{\partial x}=\frac{\partial \hat f_P}{\partial y}=0$ 
has a solution in $(\R\setminus\{0\})^2$.
Otherwise it is said to be {\it non-degenerate}.

Let $\hat f$ be a rational function given as above and let $\hat Q_i={}^t(\hat p_i,\hat q_i)$, 
$i=1,\ldots,\hat n$, be primitive covectors such that
\begin{itemize}
\item[(1)] both $\hat p_i$ and $\hat q_i$ are positive;
\item[(2)] $\varDelta(\hat Q_i;\hat f)$ is a compact face;
\item[(3)] the indices are assigned in the counter-clockwise orientation.
\end{itemize}
Let $\hat R_i={}^t(\hat r_i,\hat s_i)$, $i=1,\ldots,\hat m$, 
be primitive covectors which satisfy the following:
\begin{itemize}
\item[(1)] $\hat R_1={}^t(1,0)$ and $\hat R_{\hat m}={}^t(0,1)$;
\item[(2)] both $\hat r_i$ and $\hat s_i$ are positive for each $\hat R_i$, $i=2,\ldots,\hat m-1$;
\item[(3)] $\hat Q_i$ is contained in $\{\hat R_2,\ldots,\hat R_{\hat m-1}\}$;
\item[(4)] the indices are assigned in the counter-clockwise orientation;
\item[(5)] the determinants of the matrices $(\hat R_i,\hat R_{i+1})$, $i=1,\ldots,\hat m-1$, are $1$.
\end{itemize}
For each Cone$(\hat R_i,\hat R_{i+1})$, $i=1,\ldots,\hat m-1$, 
an affine coordinate chart $(u_i,v_i)$ is defined by the coordinate transformation
\[
   x=u_i^{\hat r_i}v_i^{\hat r_{i+1}},\quad y=u_i^{\hat s_i}v_i^{\hat s_{i+1}}.
\]
Then a real variety $Y$ is obtained by gluing these coordinate charts, which is described as
\[
   Y=U\cup\left(\bigcup_{i=2}^{\hat m-1} E(\hat R_i)\right),
\]
where $E(\hat R_i)$ is the exceptional divisor corresponding to the covector $\hat R_i$.
Let $\pi:Y\to U$ be the associated proper mapping, which is called the {\it admissible toric modification
associated with $\{\hat R_1,\ldots,\hat R_{\hat m}\}$}.
For further information about toric modifications, see~\cite{oka}.
\vspace{5mm}

Suppose that $\hat f$ has the form
\[
\hat f(x,y)=x^d(y+c)^{d'}(y^\mu \hat g(y)+x\hat h(x,y)),
\]
where $d,d'\in\Z$, $c\ne 0$, $\hat g(y)$ is the expansion of a rational function
of one variable $y$ with $\hat g(0)\ne 0$,
and $h(x,y)$ is the expansion of a rational function of two variables $(x,y)$
with $|h(0,0)|<\infty$.
Let $\varDelta^-(\hat f)$, $\varDelta^0(\hat f)$ and $\varDelta^+(\hat f)$ denote the union
of the compact faces $\varDelta(P;\hat f)$ of $\varDelta^{\text{\rm loc}}(\hat f)$ 
with $d(P;\hat f)<0$, $d(P;\hat f)=0$ and $d(P;\hat f)>0$, respectively.
Set $\ell^-(\hat f)$ and $\ell^0(\hat f)$ to be $-1$ plus the number of lattice points in the segment 
obtained by projecting $\varDelta^-(\hat f)$ and $\varDelta^0(\hat f)$ to the second axis of $\R^2$
on which $\varDelta^{\text{\rm loc}}(\hat f)$ is described, respectively.
Set $\ell^+(\hat f)=\mu-\ell^-(\hat f)-\ell^0(\hat f)$.

\begin{dfn}
The integers $\ell^+(\hat f)$, $\ell^0(\hat f)$ and $\ell^-(\hat f)$ are called
the $(+)$-, $(0)$- and $(-)$-{\it height} of $\varDelta^{\text{\rm loc}}(\hat f)$, respectively.
\end{dfn}

These heights will be used in the proof of Theorem~\ref{thm5.1}.
\vspace{5mm}

Let $f$ be a polynomial function. We first apply an admissible toric compactification $Y_1\supset \R^2$
associated with primitive covectors $\{R_1,\ldots,R_m\}$ with respect to $\varDelta(f)$.
Suppose that $f_{R_i}$ is degenerate for a face $\varDelta(R_i;f)$ in $\Gamma_\infty^-(f)$.
On $U_i$, $f$ is given as~\eqref{eq2.1}.
Let $s_1,\ldots,s_\eta$ be non-zero real roots of $g_i(v_i)=0$ and $\mu_1,\ldots,\mu_\eta$ 
their multiplicities.
For some $\xi\in\{1,\ldots,\eta\}$ with $\mu_\xi\geq 2$, which exists since $f_{R_i}$ is degenerate,
we apply the change of coordinates
\[
   (x_1,y_1)=(u_i,v_i-s_\xi).
\]
We call $(x_1,y_1)$ {\it translated coordinates}.
The polynomial function $f$ can be extended to $Y_1$ as a rational function, and is given
on the chart $(x_1,y_1)$ as
\[
f^{1}(x_1,y_1)=
x_1^{d(R_i;f)}(y_1+s_\xi)^{d(R_{i+1};f)}(y_1^{\mu_\xi}g^{1}(y_1)+x_1h^{1}(x_1,y_1)),
\]
where $g^{1}(0)\ne 0$.

Assume that we have applied admissible toric modifications $\pi_i:Y_i\to Y_{i-1}$ 
for $i=2,\ldots,\sigma$ successively.
Let $U^\sigma$ be a neighborhood of the origin on the coordinate chart $(u_\sigma, v_\sigma)$
in $Y_\sigma$ obtained after the successive toric modifications and translations of coordinates.
We call $(u_\sigma, v_\sigma)$ {\it translated coordinates} also.
Let $f^{\sigma}$ be the restriction of the pull-back of $f$ to $U^\sigma$,
which is given as 
\[
f^{\sigma}(x_\sigma,y_\sigma)=
x_\sigma^{d_\sigma}(y_\sigma+s_\xi^{\sigma})^{d'_\sigma}(y_\sigma^{\mu_\sigma} g^{\sigma}(y_\sigma)+x_\sigma
h^{\sigma}(x_\sigma,y_\sigma)),
\]
where $d_\sigma, d'_\sigma\in\Z$ with $d_\sigma<0$, $s_\xi^{\sigma}\ne 0$, $\mu_\sigma\geq 2$ and $g^{\sigma}(0)\ne 0$.
Applying an admissible toric modification $\pi_{\sigma+1}:Y_{\sigma+1}\to Y_{\sigma}$
on $U^\sigma$ with respect to $\varDelta^{\text{\rm loc}}(f^{\sigma})$,
we obtain a sequence of admissible toric modifications inductively.

We say that a sequence $Y_\tau\to\cdots\to Y_1\supset\R^2$
of successive toric modifications is {\it terminated} 
if there are no translated coordinates for further toric modifications.
Note that a sequence of successive toric modifications is terminated in finite steps.
The finiteness is proved in~\cite[Lemma~4.3]{ishikawa} for complex polynomial case and 
the same proof works for real case also.

Let $Y_\tau\to\cdots \to Y_\sigma\to\cdots\to Y_1\supset\R^2$ be 
a sequence of admissible toric modifications, which is not necessary to be terminated.
Let $\{R^{\sigma}_1,\ldots,R^{\sigma}_{m_\sigma}\}$ be the primitive covectors
for the toric modification $\pi_{\sigma+1}:Y_{\sigma+1}\to Y_\sigma$
with respect to $\varDelta^{\text{\rm loc}}(f^{\sigma})$,
containing primitive covectors $\{Q^{\sigma}_1,\ldots,Q^{\sigma}_{n_\sigma}\}$ 
orthogonal to the compact faces of $\varDelta^{\text{\rm loc}}(f^{\sigma})$.
For each $j=2,\ldots,m_\sigma-1$, on the local chart $U^\sigma_j$ in $Y_{\sigma+1}$ 
corresponding to Cone$(R^{\sigma}_j, R^{\sigma}_{j+1})$, 
the pull-back $f^{\sigma}_j$ of $f$ is given as
\[
   f^{\sigma}_j(u_{\sigma,j},v_{\sigma,j})
   =u_{\sigma,j}^{d(Q^{\sigma}_j;f^{\sigma})}
     v_{\sigma,j}^{d(Q^{\sigma}_{j+1};f^{\sigma})}
   (g^{\sigma}_j(v_{\sigma,j})
   +u_{\sigma,j}h^{\sigma}_j(u_{\sigma,j},v_{\sigma,j})).
\]

Now we define an integer $\lambda(Q^{\sigma}_j;f^{\sigma})$ by
\begin{equation}\label{eq5.2}
\lambda(Q^{\sigma}_j;f^{\sigma})
=\begin{cases}
0 & d(Q^{\sigma}_j;f^{\sigma})>0 \\
r^0(Q^{\sigma}_j;f^{\sigma}) & d(Q^{\sigma}_j;f^{\sigma})=0 \\
\sum_{\xi\in\Xi_{\sigma,j}}(\mu_\xi -1) & d(Q^{\sigma}_j;f^{\sigma})<0,
\end{cases}
\end{equation}
where
$r^0(Q^{\sigma}_j;f^{\sigma})$ is the number of non-zero real roots of 
$\frac{\partial  b_{Q^{\sigma}_j}}{dt}(t)=0$,
$\Xi_{\sigma,j}$ is the set of indices of non-zero real multiple roots of $g^{\sigma}_j(v_{\sigma,j})=0$
at which we did not apply further successive toric modifications in $Y_\tau\to\cdots\to Y_1\supset\R^2$,
and $\mu_\xi$ is the multiplicity of the root with index $\xi\in\Xi_{\sigma,j}$.
Set $\epsilon_\sigma=1$ if the $(+)$-height of $\varDelta(f^\sigma_j)$ is more than or equal to $2$,
and set $\epsilon_\sigma=0$ otherwise.

Similarly, for the primitive covectors $\{Q_1,\ldots,Q_n\}$ orthogonal to the faces of $\varDelta(f)$, 
we define
\[
\lambda(Q_i;f)
=\begin{cases}
0 & d(Q_i;f)>0 \\
r^0(Q_i;f) & d(Q_i;f)=0 \\
\sum_{\xi\in\Xi_i}(\mu_\xi -1) & d(Q_i;f)<0,
\end{cases}
\]
where $r^0(Q_i;f)$ is the number of non-zero real roots of $\frac{\partial  b_{Q_i}}{dt}(t)=0$,
$\Xi_i$ is the set of indices of real multiple roots of $g_i(v_i)=0$ in~\eqref{eq2.1}
at which we did not apply further successive toric modifications,
and $\mu_\xi$ is the multiplicity of the root with index $\xi\in\Xi_i$.
Set $\epsilon=0$ if $R^+=0$ and $\epsilon=1$ if $R^+>0$ as in Theorem~\ref{thm5.1}.

\begin{prop}\label{prop5.2}
Let $Y_\tau\to\cdots\to Y_1\supset\R^2$ be a sequence of successive toric modifications
which is terminated. Then
\[
|B_f|\leq |\Sigma_f|+\epsilon+\sum_{i=1}^n\lambda(Q_i;f)
+\sum_{\sigma}\left(\epsilon_\sigma+\sum_{j=1}^{n_\sigma}\lambda(Q^{\sigma}_j;f^{\sigma})\right),
\]
where $\sigma$ runs over all indices of translated coordinates 
appearing in the successive toric modifications.
\end{prop}

\begin{proof}
Since the sequence of successive toric modifications is terminated, 
$\Xi_{\sigma,j}=\emptyset$ for any $(\sigma,j)$.

We first check the contribution of the faces $\varDelta(Q^\sigma_j;f^\sigma)$
with $d(Q^\sigma_j;f^\sigma)>0$ to $|B_f|$.
Consider the variety $Y_{\sigma+1}$ obtained by an admissible toric modification 
$\pi_{\sigma+1}:Y_{\sigma+1}\to Y_\sigma$ with respect to $\varDelta^{\text{\rm loc}}(f^\sigma)$.
Let $\ell^+(f^\sigma)$ be the $(+)$-height of $\varDelta^{\text{loc}}(f^\sigma)$.
Suppose that $\ell^+(f^\sigma)\geq 2$.
If there are cleaving families near $E(Q^\sigma_j)$ with $d(Q^\sigma_j;f^\sigma)>0$  
then their limits correspond to the same value in $B_f$.
Hence its contribution is at most $1$. If $\ell^+(f^\sigma)=0$ then there is no contribution.
If $\ell^+(f^\sigma)=1$ and there is no face $\varDelta(P;f^\sigma)$ with 
$d(P;f^\sigma)>0$ then there is no contribution also.
Suppose that $\ell^+(f^\sigma)=1$ and there is such a face, say $\varDelta(Q^\sigma_{j_0};f^\sigma)$.
Then, as shown in Figure~\ref{fig9}, we see that the fibration of nearby fibers
passing near $E(Q^\sigma_{j_0})$ is trivial. 
Hence the contribution of the face $\varDelta(Q^\sigma_{j_0};f^\sigma)$ is $0$. 
There is no vanishing family in a neighborhood of $E(Q^\sigma_j)$ 
with $d(Q^\sigma_j;f^\sigma)>0$ in $Y_{\sigma+1}$.
Thus the contribution of the faces $\varDelta(Q^\sigma_j;f^\sigma)$ with 
$d(Q^\sigma_j;f^\sigma)>0$ is at most $\epsilon_\sigma$.
\begin{figure}[htbp]
   \centerline{\input{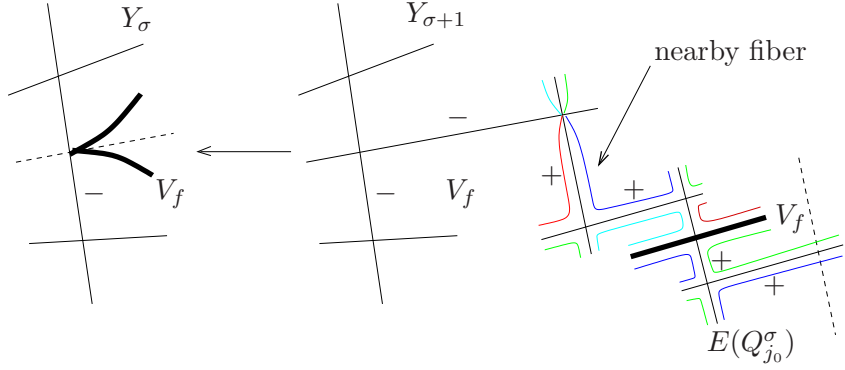}}
  \caption{The triviality of the fibration in the case $\ell^+(f^\sigma)=1$.\label{fig9}}
\end{figure}

Next we check the contribution of the faces $\varDelta(Q^\sigma_j;f^\sigma)$ 
with $d(Q^\sigma_j;f^\sigma)=0$.
The number of values appearing as their limits is at most the number of non-zero real roots 
of $\frac{db_{Q^\sigma_j}}{dt}(t)=0$.
Hence the contribution is at most $r^0(Q^\sigma_j;f^\sigma)$.

The same observation can be applied to neighborhoods of the divisors $E(Q_i)$, $i=1,\ldots,n$,
and we have the upper bound $\epsilon+\sum_{i=1}^n\lambda(Q_i;f)$ of the contribution.
This completes the proof.
\end{proof}

\begin{rem}
The phenomenon having the triviality of the fibration in the case $\ell^+=1$, shown in Figure~\ref{fig9},
appears in complex polynomial case also. A face $\varDelta(P;f^\sigma)$ with 
$d(P;f^\sigma)>0$ in the case $\ell^+=1$ is called a {\it stable boundary face} 
in~\cite[Definition~5.5]{ishikawa}. 
The exception of such a face is pointed out, for instance, in~\cite[Example 2.9~(3)]{blh} also.
\end{rem}

\begin{proof}[Proof of Theorem~\ref{thm5.1}]
Let $Y_\tau\to\cdots\to Y_1\supset\R^2$ be a sequence of successive toric modifications
which is not terminated.
Set
\[
   \Lambda_\tau=\epsilon+\sum_{i=1}^n\lambda(Q_i;f)
   +\sum_{\sigma=1}^\tau\left(\epsilon_\sigma+\sum_{j=1}^{n_\sigma}\lambda(Q^{\sigma}_j;f^{\sigma})\right).
\]
We first prove that this sum does not increase after an admissible toric modification
$\pi_{\tau+1}:Y_{\tau+1}\to Y_\tau$, i.e., 
prove the inequality $\Lambda_{\tau+1}\leq\Lambda_\tau$.

Apply a toric modification $\pi_{\tau+1}$ at the origin of translated coordinates 
$(x_\tau, y_\tau)$.
The pull-back $f^{\tau+1}=\pi_{\tau+1}^*f^\tau$ of $f$ has the form
\[
f^{\tau+1}(x_{\tau+1},y_{\tau+1})=x_{\tau+1}^{d_{\tau+1}}
(y_{\tau+1}+s_\xi)^{d'_{\tau+1}} 
(y_{\tau+1}^{\mu_\xi}g^{\tau+1}(y_{\tau+1})+x_{\tau+1}h^{\tau+1}(x_{\tau+1},y_{\tau+1})),
\]
where $d_{\tau+1}, d'_{\tau+1}\in\Z$ with $d_{\tau+1}<0$, 
$s_\xi\ne 0$, $\mu_\xi\geq 2$ and $g^{{\tau+1}}(0)\ne 0$.
Let $\ell^+$, $\ell^0$ and $\ell^-$ denote the
$(+)$-, $(0)$- and $(-)$-heights of $\varDelta^{\text{\rm loc}}(f^{\tau+1})$, respectively.
Note that $\ell^++\ell^0+\ell^-\leq \mu_\xi$.

We will prove that the total contribution of the faces $\varDelta(Q^{\tau+1}_j;f^{\tau+1})$ 
to $\Lambda_{\tau+1}$ is at most $\ell:=\ell^++\ell^0+\ell^--1$.
From $\varDelta^{\text{\rm loc}}(f^{\tau+1})$, we see that the contribution 
$\sum_{\xi\in\Xi_{\tau+1,j}}(\mu_\xi -1)$ in~\eqref{eq5.2} in the case 
$d(Q^{\tau+1}_j;f^{\tau+1})<0$ is at most $\ell^--1$.
The contribution in the case $d(Q^{\tau+1}_j;f^{\tau+1})=0$ is at most $\ell^0$ and 
the contribution in the case $d(Q^{\tau+1}_j;f^{\tau+1})>0$ is $\epsilon_{\tau+1}$.
Hence if $\ell^->0$ then the total contribution is at most $\ell$.
Suppose that $\ell^-=0$. If $\ell^+=0$ then the face $\varDelta(P;f^{\tau+1})$
with $d(P;f^{\tau+1})=0$ contains the origin $(0,0)$ and 
the contribution in the case $d(Q^{\tau+1}_j;f^{\tau+1})=0$ becomes at most $\ell^0-1$.
Hence the total contribution is at most $\ell$.
If $\ell^+\geq 2$ then 
the contribution in the case $d(Q^{\tau+1}_j;f^{\tau+1})>0$ becomes at most $\ell^+-1$,
and hence the total contribution is also at most $\ell$.
If $\ell^+=1$ then $\epsilon_{\tau+1}=0$, i.e., the contribution in the case $d(Q^{\tau+1}_j;f^{\tau+1})>0$ is $0$.
Hence the total contribution is at most $\ell$.
Since $\ell:=\ell^++\ell^0+\ell^--1\leq \mu_\xi-1$, we have
\[
   \Lambda_{\tau+1}\leq \epsilon+\sum_{i=1}^n\lambda(Q_i;f)
   +\sum_{\sigma=1}^\tau\left(\epsilon_\sigma+\sum_{j=1}^{n_\sigma}\lambda(Q^{\sigma}_j;f^{\sigma})\right)
   -(\mu_\xi-1)+\ell \leq \Lambda_\tau.
\]
Thus $\Lambda_\tau$ does not increase.
By the same argument, we have $\Lambda_1\leq\Lambda_0:=\epsilon+\sum_{i=1}^n\lambda(Q_i;f)$.

Suppose that a sequence of successive toric modifications is terminated at $\tau=\tau_0$.
By Proposition~\ref{prop5.2} we have $|B_f|\leq |\Sigma_f|+\Lambda_{\tau_0}$.
We then apply the inequality $\Lambda_{\tau+1}\leq\Lambda_\tau$ inductively:
\[
\begin{split}
|B_f|&\leq |\Sigma_f|+\Lambda_{\tau_0}
\leq |\Sigma_f|+\Lambda_{\tau_0-1}
\leq \cdots\leq |\Sigma_f|+\Lambda_1 \\
&\leq |\Sigma_f|+\Lambda_0=
|\Sigma_f|+\epsilon+R^0+\sum_{\varDelta(P;f)\in\Gamma_\infty^-(f)}\mu(P;f).
\end{split}
\]
This completes the proof.
\end{proof}

\begin{rem}
Let $Y_\tau\to\cdots\to Y_1\supset\R^2$ be a sequence of successive toric modifications
which is terminated. Then a vanishing family appears only in a neighborhood of 
a divisor $E(Q^\sigma_j)$ with $d(Q^\sigma_j;f^\sigma)=0$.
\end{rem}


\begin{thebibliography}{99}

\bibitem{blh}
E.~Artal Bartolo, I,~Luengo, A. Melle-Hern\'andez,
{\it High school algebra of the theory of dicritical divisors:
Atypical fibers for special pencils and polynomials}
J. Algebra Appl. {\bf 14} (2015), no. 9, 1540009, 26 pp. 

\bibitem{bp}
A.~Bodin, A.~Pichon, 
{\it Meromorphic functions, bifurcation sets and fibred links},
Math. Res. Lett. {\bf 14} (2007), no. 3, 413--422.

\bibitem{cp}
M. Coste, M.J. de la Puente, 
{\it Atypical values at infinity of a polynomial function on the real plane: an erratum, and an algorithmic criterion},
J. Pure Appl. Algebra {\bf 162} (2001), no. 1, 23--35. 

\bibitem{dt}
L.G. Dias, M. Tib\u{a}r, 
{\it Detecting bifurcation values at infinity of real polynomials},
Math. Z. {\bf 279} (2015), no. 1-2, 311--319.

\bibitem{HL1984} 
{H.V. H\`a and D.T. L\^e,}
{\it Sur la topologie des polyn\^omes complexes,} Acta Math. Vietnam. {\bf 9},
(1984) 21--32.


\bibitem{HaHV1989-2} 
H.V. H\`a and L.A. Nguy\^en,
{\it Le comportement g\'eoom\'etrique \`a l'infini des polyn\^omes de deux
variables complexes,} C. R. Acad. Sci., Paris, S\'erie I, {\bf 309}, No. 3,
(1989) 183--186.

\bibitem{hn}
H.V. H\`a and T.T. Nguy\^en,
{\it Atypical values at infinity of polynomial and rational functions on an algebraic surface in $\R^n$},
Acta Math. Vietnamica {\bf 36} (2011) no. 2, 537--553.



\bibitem{ishikawa}
M. Ishikawa, 
{\it The bifurcation set of a complex polynomial function of two variables and the Newton polygons of singularities at infinity}, J. Math. Soc. Japan {\bf 54} (2002), no. 1, 161--196. 

\bibitem{Kouchnirenko1976} 
A. G. Kouchnirenko, 
{\it Polyhedres de Newton et nombre de Milnor}, Invent. Math. {\bf 32} (1976), 1--31. 

\bibitem{lo}
V.T.~Le, M.~Oka,
{\it Estimation of the number of the critical values at infinity of a polynomial function 
$f:\mathbf C^2\to\mathbf C$}, Publ. Res. Inst. Math. Sci. {\bf 31} (1995), no. 4, 577--598.

\bibitem{nz}
A. Nemethi, A. Zaharia, 
{\it Milnor fibration at infinity},
Indag. Math. (N.S.) {\bf 3} (1992), no. 3, 323--335.

\bibitem{oka}
M. Oka, Non-degenerate Complete Intersection Singularity,
Actualit\'es Math., Hermann, 1998.

\bibitem{son}
T.S. Pham, 
{\it On the topology of the Newton boundary at infinity},
J. Math. Soc. Japan {\bf 60} (2008), no. 4, 1065--1081.

\bibitem{thang}
T.T.~Nguyen,
{\it Bifurcation set, M-tameness, asymptotic critical values and Newton polyhedrons},
Kodai Math. J. {\bf 36} (2013), no. 1, 77--90.

\bibitem{Suzuki1974}
M. Suzuki,
{\it Propri\'et\'es topologiques des polyn\^omes de deux variables complexes,
et automorphismes alg\'ebriques de l'espace  ${\Bbb C}^2$,}  J. Math. Soc.
Japan, {\bf 26}, (1974) 241--257.


\bibitem{tz}
M. Tib\u{a}r, A. Zaharia,
{\it Asymptotic behaviour of families of real curves},
Manuscripta Math. {\bf 99} (1999), 383--393.

\bibitem{zaharia}
A. Zaharia,
{\it On the bifurcation set of a polynomial function and Newton boundary, II},
Kodai Math. J. {\bf 19} (1996), 218--223.

\end{thebibliography}
\end{document}